\documentclass[conference]{IEEEtran}
\usepackage{amssymb}
\usepackage{amsmath}
\usepackage[hmargin=2cm,vmargin={3.5cm,4cm}]{geometry}
\newtheorem{te}{Theorem}[section]

\newtheorem{definition}[te]{Definition}
\newtheorem{os}[te]{Remark}

\newtheorem{lem}[te]{Lemma}

\begin{document}

    \title{Fractional Klein--Gordon equation for linear dispersive phenomena: analytical
     methods and applications}

    \author{\IEEEauthorblockN{Roberto Garra}
    \IEEEauthorblockA{Dipartimento di Scienze di Base e Applicate per l'Ingegneria,\\
    ``Sapienza'' Universit\`a di Roma, Italy\\
     Email: roberto.garra@sbai.uniroma1.it}
     \IEEEauthorblockN{Enzo Orsingher}
    \IEEEauthorblockA{Dipartimento di Scienze Statistiche\\
    ``Sapienza'' Universit\`a di Roma, Italy\\
    Email: enzo.orsingher@uniroma1.it}\and
    \IEEEauthorblockN{Federico Polito}
    \IEEEauthorblockA{Dipartimento di Matematica ``G. Peano''\\
    Universit\`a degli Studi di Torino, Italy\\
    Email: federico.polito@unito.it}}


    \date{\today}

    \maketitle

    \begin{abstract}
        In this paper we discuss some explicit results
        related to the fractional Klein--Gordon equation involving fractional
        powers of the D'Alembert operator. By means of a space-time transformation,
        we reduce the fractional Klein--Gordon equation to a case of
        fractional hyper-Bessel equation. We find an explicit
        analytical
        solution by using the McBride theory of fractional powers of hyper-Bessel operators.
        These solutions are expressed in terms of multi-index
        Mittag-Leffler functions studied by Kiryakova and Luchko
        \cite{Kir-Luch}.
        A discussion of these results within the framework of
        linear dispersive wave equations is provided. We also
        present exact solutions of the fractional Klein--Gordon equation in the higher dimensional cases.
        Finally, we suggest a method of finding travelling wave
        solutions of the nonlinear fractional Klein--Gordon equation
        with power law nonlinearities.
    \end{abstract}

    \section{Introduction}

       In this paper we study the following fractional Klein--Gordon equation
       \begin{equation}\label{1}
       \left(\frac{\partial^2}{\partial t^2}-c^2\Delta
               \right)^{\alpha}u_{\alpha}(\mathbf{x},t)=-\lambda^2
               u_{\alpha}(\mathbf{x},t),
       \end{equation}
       where $\alpha \in (0,1]$, and $\mathbf{x}\in
       \mathbb{R}^N$.
        Hence we consider a space-time fractional order operator, that is a fractional power of the D'Alembert
        operator.
        In order to analyze the travelling wave-type solutions of \eqref{1} the main trick
        is based on the space-time transformation
        \begin{equation}\nonumber
        w= \sqrt{c^2t^2-\sum_{k=1}^N x_k^2 },
        \end{equation}
        which converts \eqref{1} into a case of the fractional
        hyper-Bessel equation
        \begin{equation}
            \label{ki}
            \left(\frac{d^2}{dw^2}+\frac{N}{w}\frac{d}{dw}\right)^{\alpha}u_\alpha(w)
            = -\frac{\lambda^2}{c^{2\alpha}}u_\alpha(w).
        \end{equation}
        In order to treat \eqref{ki} we will use the theory developed by
        A.C. McBride in a series of papers on fractional power of
        hyper-Bessel-type operators. Here we recall one of his results
        by showing that the explicit representation of a general hyper-Bessel-type operator is
        given as a product of Erd\'elyi--Kober fractional integrals. This
        fact is also at the basis of the generalized fractional calculus developed by Kiryakova
        \cite{Kir-1994}.
          By means of this theory we find in an explicit form
        travelling wave solutions of the fractional Klein--Gordon
        equation. We consider both the one-dimensional and higher-dimensional
        cases. Similar results on fractional Klein--Gordon-type equations have been recently
        discussed in \cite{us}, where an application to the fractional telegraph-type processes
        has been investigated. A similar approach was adopted by Garra et al. \cite{Garra-FCAA} for the
        study of the fractional relaxation equation with time-varying coefficients.

        In view of these results, we also study the nonlinear
        fractional
        Klein--Gordon equation with power law nonlinearities. By recurring to
        the general theory, we are able to find in explicit form some particular
        solutions also in the nonlinear case.
        The main aim of this paper is to give new mathematical tools to
        solve linear and nonlinear space-time fractional equations that
        are strictly related to the propagation of linear dispersive
        waves. Moreover we show the way to treat fractional-Bessel
        equations that have wide applications in different fields of
        physics.

    \section{Fractional hyper-Bessel operators}
        In this section we briefly recall some useful results on
        the fractional power of hyper-Bessel-type operators.
        We refer to the theory developed by McBride in a
        series of papers \cite{mc2, mc1, mc}.

        The hyper-Bessel operator considered in
        \cite{mc} is defined as
        \begin{equation}
            \label{L}
            L=x^{a_1}Dx^{a_2}\dots x^{a_n}Dx^{a_{n+1}},
        \end{equation}
        where $n$ is a positive integer number, $a_1,\dots, a_{n+1}$ are complex numbers and $D=d/dx$.
        Hereafter we assume that the coefficients $a_j$, $j=1,\dots, n+1$, are real numbers.
        The operator $L$ was first introduced and studied, also with
        its fractional powers, by Dimovski \cite{Dim} and served as
        a base for the generalized fractional calculus in Kiryakova
        \cite{Kir-1994}. In this book the whole chapter 3 is devoted
        to the hyper-Bessel operators, the solutions of differential
        equations involving it and to the development of its theory
        in terms of products of Erd\'elyi-Kober operators.
        Fractional power of second order version of
        \begin{align*}
            L_{B_n}=x^{-n}\underbrace{x\frac{d}{dx}x\frac{d}{dx}\dots x\frac{d}{dx}}_{\text{$n$ times}}.
        \end{align*}
        are dealt with in section III and IV below.
        
        By using operational methods, the integer power of the operator $L$
        can be explicitly given in terms of a product of
        Erd\'elyi--Kober fractional derivatives
        (for further details see \cite{mc}, pag. 527 and \cite{Kir-1994} pag.59-60).
        In what follows, we use the notations adopted in McBride
        works.
        Let us define the coefficients
        \begin{align*}
            &a=\sum_{k=1}^{n+1}a_k, \qquad m= |a-n|,\\
            &b_k= \frac{1}{m}\left(\sum_{i=k+1}^{n+1}a_i+k-n\right),
            \quad k=1, \dots,
            n.
        \end{align*}

        It is possible to prove the following result which was
        first formulated in \cite{mc}
        \begin{lem} \label{duepuntouno}
            Let $r$ be a positive integer, $a<n$,
            \begin{align*}
                b_k\in A_{p,\mu,m}:=&\{\eta \in \mathbb{C}
                \colon \Re(m\eta+\mu)+m\neq 1/p-ml, \\
                &\nonumber\: l=0, 1, 2,\dots\}, \quad k=1,\dots, n,
            \end{align*}
            where $(p,\mu)\in [1,+\infty)\times \mathbb{C}$.
            Then
            \begin{equation}
                L^r f= m^{nr}x^{-mr}\prod_{k=1}^n I^{b_k,-r}_m f,
            \end{equation}
            where, for $\alpha >0$ and $\Re(m\eta+\mu)+m > 1/p$
            \begin{equation}
                \nonumber
                I_m^{\eta, \alpha}f=
                \frac{x^{-m\eta-m\alpha}}{\Gamma(\alpha)}\int_0^x(x^m-u^m)^{\alpha-1}u^{m\eta}f(u)\, d(u^m),
            \end{equation}
            where the above notation is used for the Erd\'elyi--Kober fractional integrals; and for $\alpha\leq 0$
            \begin{equation}
                \nonumber
                I_m^{\eta, \alpha}f=(\eta+\alpha+1)I_m^{\eta, \alpha+1}f+\frac{1}{m} I_m^{\eta, \alpha+1}
                \left(x\frac{d}{dx}f\right),
            \end{equation}
            which is pratically an Erd\'elyi--Kober derivative in
            the sense of Kiryakova \cite{Kir-1994}.
        \end{lem}
        The couple of parameters $(p, \mu)$ is related to the functional space to which $f$ belongs \cite{mc}.
        The fractional generalization $L^{\alpha}$ of the operator $L$ is
        consequently defined as follows
         (see \cite{mc}, pag. 527).

        \begin{definition}
            Let $m= n-a>0$, $\alpha\in \mathbb{R}$, $b_k\in A_{p,\mu,m}$, for $k=1,\dots, n$.
            Then,
            \begin{equation}
                \label{pot}
                L^{\alpha}f=m^{n\alpha}x^{-m\alpha}\prod_{k=1}^{n}I_{m}^{b_k,-\alpha}f.
            \end{equation}
        \end{definition}
        Note that, for $n=1$, $a_1=a_2=0$ and $\alpha>0$, equation \eqref{pot} coincides with the Riemann--Liouville
        fractional derivative of order $\alpha$ (see \cite{pod} Section 2.3).
        Moreover, we observe that the topic of fractional Bessel equations has been
        considered in recent papers with a different approach (see e.g. \cite{fca} and the references therein).
        An application to the description of corneal topography has been also suggested in \cite{polo}.
        A complete study of different approaches to fractional Bessel
        equations and their
        applications should be object of a further research.

        The following lemma plays a relevant role for the next
        calculations.
        \begin{lem}
            \label{brunello}
            Let be $\eta+\frac{\beta}{m}+1 >0$, $m\in \mathbb{N}$, $\alpha\in \mathbb{R}$, we have
            that
            \begin{equation}
                I_m^{\eta,\alpha}x^{\beta}=\frac{\Gamma\left(\eta+\frac{\beta}{m}+1\right)}
                {\Gamma\left(\alpha+\eta+1+\frac{\beta}{m}\right)}x^{\beta}.
            \end{equation}
        \end{lem}

    \section{Fractional Klein--Gordon equation}
        \label{piccolo-gordon}

        \subsection{The one-dimensional case}
            Let us consider the following fractional Klein--Gordon equation
            \begin{align}
                \label{KG}
                &\left(\frac{\partial^2}{\partial t^2}-c^2\frac{\partial^2}{\partial x^2}
                \right)^{\alpha}u_{\alpha}(x,t)=-\lambda^2 u_{\alpha}(x,t),\\
                \nonumber &x\in \mathbb{R}, \:
                t\geq 0, \: \alpha \in (0,1].
            \end{align}
            The classical Klein--Gordon equation ($\alpha=1$)
            emerges from the quantum relativistic energy equation. It
            is also used in the analysis of wave propagation
            in linear dispersive media (see, for example, \cite{main}).
            The fractional Klein--Gordon equation was recently studied in the
            context of nonlocal quantum field theory, within the stochastic
            quantization approach (see \cite{li} and the references therein).
            The fractional power of D'Alembert operator has been considered
            by \cite{Bol} and \cite{Lamb}, with different approaches.

            The transformation
            \begin{align}
                \begin{cases}
                    \nonumber z_1 = ct+x,\\
                    \nonumber z_2 = ct-x,
                \end{cases}
            \end{align}
            reduces \eqref{KG} to the form
            \begin{equation}
                \label{k1}
                \left(4c^2\frac{\partial}{\partial z_1}\frac{\partial}{\partial z_2}
                \right)^{\alpha}u_{\alpha}(z_1,z_2)= -\lambda^2 u_{\alpha}(z_1,z_2).
            \end{equation}
            The partial differential equation \eqref{k1} involves
            in fact Riemann--Liouville fractional derivatives with respect to the
            variables $z_1$ and $z_2$.
            The further transformation
            $w= \sqrt{z_1\,z_2}$
            gives the fractional Bessel equation
            \begin{equation}
                \label{k2}
                \left(\frac{d^2}{dw^2}+\frac{1}{w}\frac{d}{dw}\right)^{\alpha}u_\alpha(w)
                = -\frac{\lambda^2}{c^{2\alpha}}u_\alpha(w).
            \end{equation}
            The Bessel operator
            \begin{align*}
                L_{B}=\frac{d^2}{dw^2}+\frac{1}{w}\frac{d}{dw}
            \end{align*}
            appearing in
            \eqref{k2} is a special case of $L$, when $n=2$, $a_1=-1$, $a_2= 1$,
            $a_3=0$. By definition \eqref{pot} and Lemma \ref{duepuntouno} we have that $m= 2$,
            $b_1=b_2=0$ and thus
            \begin{equation}
                \label{k3}
                (L_B)^{\alpha}f(w)=4^{\alpha}w^{-2\alpha}I_2^{0,-\alpha}I_2^{0,-\alpha}f(w).
            \end{equation}

            We are now ready to state the following
            \begin{te}
                \label{gallico}
                Let $\alpha\in(0,1]$, the fractional equation
                \begin{align}\label{ciccio}
                    (L_B)^{\alpha}u_{\alpha}(w)= -\frac{\lambda^2}{c^{2\alpha}}u_{\alpha}(w),
                \end{align}
                is satisfied by
                \begin{equation}
                    \label{k5}
                    u_{\alpha}(w)=w^{2\alpha-2}\sum_{k=0}^{\infty}(-1)^k \left(
                    \frac{\lambda}{2^\alpha c^{\alpha}}\right)^{2k}
                    \frac{w^{2\alpha k}}{[\Gamma(\alpha k+\alpha)]^2}.
                \end{equation}
            \end{te}

            \begin{proof}   Let $\beta >0$, we have that
                \begin{align}
                    \label{pri}
                    (L_B)^{\alpha}w^{\beta} & =4^{\alpha}w^{-2\alpha}I_2^{0,-\alpha}I_2^{0,-\alpha}w^{\beta}\\
                    \nonumber & = 4^{\alpha}\left[\frac{\Gamma\left(\frac{\beta}{2}+1\right)}
                    {\Gamma\left(1-\alpha+\frac{\beta}{2}\right)}\right]^2 w^{\beta-2\alpha}.
                \end{align}
                By applying now the operator $(L_B)^{\alpha}$ to the function \eqref{k5}
                we obtain (since $\beta = 2\alpha k+2\alpha-2$)
                \begin{align}
                    \nonumber &(L_B)^{\alpha}\left(w^{2\alpha-2}\sum_{k=0}^{\infty}(-1)^k
                    \left(\frac{\lambda}{2^\alpha c^{\alpha}}w^{\alpha}\right)^{2k}
                    \frac{1}{[\Gamma(\alpha k+\alpha)]^2}\right)\\
                    \nonumber &=4^{\alpha}\sum_{k=0}^{\infty}(-1)^k
                    \left(\frac{\lambda}{2^\alpha c^{\alpha}}\right)^{2k}
                    \frac{w^{2\alpha k-2}}{[\Gamma(\alpha
                    k)]^2}= -\frac{\lambda^2}{c^{2\alpha}}u_{\alpha}(w),
                \end{align}
                as claimed.
            \end{proof}

            \begin{os}
                Let us note that our solution to equation \eqref{ciccio}
                expressed in terms of the power series \eqref{k5} can also
                be written by recurring to the multi-index Mittag-Leffler
                functions of Kiryakova and Luchko \cite{Kir-Luch}, that is
                defined as follows
                \begin{equation}
                    E^{(n)}_{(\alpha_i)^n, (\mu_i)^n}(z)= \sum_{k=0}^{\infty}\frac{z^k}{\Gamma(\alpha_1 k+\mu_1)
                    \dots \Gamma(\alpha_n k+\mu_n)}.
                \end{equation}
                Namely we have that \eqref{k5} can be written as
                \begin{equation}
                    u_{\alpha}(w)=w^{2\alpha-2}E^{(2)}_{(\alpha, \alpha), (\alpha,
                    \alpha)}\left[-\left(\frac{\lambda w^{\alpha}}{2^{\alpha}c^{\alpha}}\right)^2\right]
                \end{equation}
            \end{os}

            Returning to the original problem, the equation \eqref{KG}
            admits the solution
            {\small \begin{align}\nonumber
                &u_{\alpha}(x,t)=(c^2t^2-x^2)^{\alpha-1}\sum_{k=0}^{\infty}(-1)^k
                \frac{\lambda^{2k}}{(2c)^{2\alpha k}}\frac{\left(c^2t^2-x^2\right)^{\alpha
                k}}{[\Gamma(\alpha k+\alpha)]^2}\\
                \nonumber &=(c^2t^2-x^2)^{\alpha-1}E^{(2)}_{(\alpha, \alpha), (\alpha,
                \alpha)}\left[-\left(\frac{\lambda
                \left(c^2t^2-x^2\right)^{\alpha/2}}{2^{\alpha}c^{\alpha}}\right)^2\right]
            \end{align}}
            which for $\alpha =1$, reduces to the Bessel function of the first kind
            \begin{align}\label{n1}
                u_1(x,t)=J_0\left(\frac{\lambda}{c}\sqrt{c^2t^2-x^2}\right), \qquad |x|<ct.
            \end{align}
            \begin{os}
                We observe that within a similar approach, some
                particular solutions of the fractional wave equation
                with a source term of the type
                \begin{equation}
                \left(\frac{\partial^2}{\partial t^2}-c^2\Delta
                \right)^{\alpha}u_{\alpha}(\mathbf{x},t)= f(x,t),
                \end{equation}
                can be simply achieved. This kind of fractional
                generalization of the wave equation is new and can be of
                interest for the applications in the fractional approach
                to the electromagnetic theory (see e.g. \cite{rosa} and
                \cite{tar}).
            \end{os}

        \subsection{Relation with the linear damped wave equation}

            We recall that the linear damped wave equation for
            waves propagating on an elastically supported string,
            when the string motion is damped by air friction, has
            the form
            \begin{equation}\label{tel}
                \left(\frac{\partial^2}{\partial t^2}-\frac{\partial^2}{\partial
                x^2}+2\sigma \frac{\partial}{\partial
                t}\right) u= -u,
            \end{equation}
            where $\sigma$ is the damping coefficient. It can be proved that by using the transformation
            \begin{equation}\label{dam}
                u(x,t)= e^{-\sigma t}v(x,t),
            \end{equation}
            equation \eqref{tel} is trasformed into the linear Klein--Gordon equation
            \begin{equation}\label{tel1}
                \left(\frac{\partial^2}{\partial t^2}-\frac{\partial^2}{\partial
                x^2}\right) v(x,t)= (\sigma^2-1)v(x,t),
            \end{equation}
            when $\sigma^2<1$. For $\sigma^2>1$, we obtain the Helmholtz equation
            which is strictly related to the telegraph equation (see e.g. \cite{us}).
            In our case we consider a space-time-fractional operator in
            the linear Klein--Gordon equation. From the point of view of
            the applications to the propagation of damped waves, our
            idea is to take into account damping effects in the
            classical way, i.e.\ with an exponential damping term such as in
            \eqref{dam} and fractional effects in the wave
            propagation by directly generalizing the
            linear Klein--Gordon equation \eqref{tel1}.

        \subsection{Higher-dimensional case}
            Higher dimensional fractional
            Klein--Gordon equations can be analyzed in a similar way. Let us consider the
            $N$-dimensional fractional Klein--Gordon equation, i.e.\
            \begin{align}
               \label{ddim}
               &\left(\frac{\partial^2}{\partial t^2}-c^2\Delta \right)^{\alpha}u_\alpha (\mathbf{x},t)
               =-\lambda^2 u_\alpha(\mathbf{x},t), \\
               \nonumber & \qquad \alpha \in (0,1], \: \mathbf{x}\in \mathbb{R}^N.
            \end{align}
            By means of the transformation
            \begin{align*}
               w =\left(c^2t^2-\sum_{k=1}^N x_k^2 \right)^{1/2},
            \end{align*}
            where $x_k$ is the $k$-th coordinate of the $N$-dimensional vector $\mathbf{x}$, we transform \eqref{ddim} into
            \begin{equation}
               \label{Lddi}
               \left(\frac{d^2}{dw^2}+\frac{N}{w}\frac{d}{dw}\right)^{\alpha}u_\alpha(w)
               =-\frac{\lambda^2}{c^{2\alpha}}u_\alpha(w).
            \end{equation}
            The operator appearing in \eqref{Lddi} can be considered again as a
            specific case of the operator \eqref{L} with $a_1=-N$, $a_2=N$,
            $a_3=0$, $a=0$, $n=m=2$, $b_1=\frac{N-1}{2}$ and $b_2=0$. Hence, from \eqref{pot} we have
            that
            \begin{equation}
              \nonumber \left(\frac{d^2}{dw^2}+\frac{N}{w}\frac{d}{dw}\right)^{\alpha}u_\alpha (w)=4^{\alpha}w^{-2\alpha}
              I_2^{0,-\alpha}I_2^{\frac{N-1}{2}, -\alpha}u_\alpha (w).
            \end{equation}
            By using arguments similar to those of the previous section, we can prove
            the following
            \begin{te}
                \label{parafarmacia}
                A solution to the $N$-dimensional fractional Klein--Gordon equation \eqref{ddim}, is given by
                {\small \begin{align}
                   \nonumber &u_\alpha (\mathbf{x},t)= \sum_{k=0}^{\infty}\left(\frac{\lambda}{2^\alpha c^{\alpha}}\right)^{2k}
                   \frac{(-1)^k\left(c^2t^2-\sum_{k=1}^N x_k^2 \right)^{\alpha k+\alpha-1}}
                   {\Gamma(\alpha k+\alpha+\frac{N-1}{2})
                   \Gamma(\alpha k+\alpha)}\\
                   &=\left(c^2t^2-\sum_{k=1}^N x_k^2
                   \right)^{\alpha-1}\\
                   \nonumber & \times E^{(2)}_{(\alpha, \alpha), (\alpha,
                   \alpha+\frac{N-1}{2})}\left[-\left(\frac{\lambda \left(c^2t^2-\sum_{k=1}^N x_k^2
                   \right)^{\alpha/2}
                   }{2^\alpha c^{\alpha}}\right)^2\right].
                \end{align}}
            \end{te}
            We observe that for $\alpha =1$, the solution of Theorem 3.2 reduces to the Bessel function
            \begin{equation}\nonumber
            u_1(\mathbf{x},t)=\frac{J_{\frac{N-1}{2}}\left(\frac{\lambda}{c}
            \sqrt{c^2t^2-\sum_{k=1}^N x_k^2}\right)}{\left(\sqrt{c^2t^2-\sum_{k=1}^N x_k^2}\right)^{N-1}}.
            \end{equation}
            For $N=1$ we retrieve result \eqref{n1}.

        \section{The nonlinear case}

            Here we consider the one-dimensional nonlinear fractional Klein--Gordon
            equation with power law nonlinearity,
            \begin{align}\label{nln}
                &\left(\frac{\partial^2}{\partial t^2}-c^2\frac{\partial^2}{\partial x^2}
                \right)^{\alpha}u_{\alpha}(x,t)= \lambda u_{\alpha}^s(x,t),\\
                \nonumber & \qquad x\in \mathbb{R}, \:
                t\geq 0, \: \alpha \in (0,1], \lambda \in \mathbb{R}, s\neq 1.
            \end{align}
            The higher dimensional case can be handled in a similar way.
            In the recent literature, some attempts to find specific
            solutions to the nonlinear fractional Klein--Gordon equation
            were discussed (see e.g.\ \cite{Bale}). However, these papers
            are based on the application of approximate methods such as
            the homotopy perturbation method and related to a
            different formulation of the fractional
            Klein--Gordon equation. In view of the previous discussion,
            we are going to find an explicit travelling wave solution of
            \eqref{nln}.
            \begin{te}
                A travelling wave solution of \eqref{nln} is given by
                \begin{align}
                u_{\alpha}(x,t)=&\left(\frac{4^{\alpha}}{\lambda}\left[\frac{\Gamma\left(1+\frac{\alpha}{1-s}\right)}
                        {\Gamma\left(1-\alpha+\frac{\alpha}{1-s}\right)}\right]^2\right)^{\frac{1}{s-1}}\\
                        \nonumber&\times\left(c^2t^2- x^2
                \right)^{\alpha/(1-s)}.
                \end{align}
                \end{te}
                \begin{proof}
                We are going to study exact solutions in the travelling wave
                form $u_{\alpha}(\sqrt{c^2t^2-x^2})$.
                By means of the transformation
                \begin{align*}
                    w =\left(c^2t^2-x^2\right)^{1/2},
                \end{align*}
                we transform equation \eqref{nln} in
                \begin{equation}
                \left(\frac{d^2}{dw^2}+\frac{1}{w}\frac{d}{dw}\right)^{\alpha}u_\alpha(w)
                    = \lambda u_\alpha^n(w).
                \end{equation}
                Assuming that the solution we are searching is in the form
                \begin{equation}\label{solu}
                    u_\alpha(w)= k w^\beta,
                \end{equation}
                where $\beta$ and $k$ are real parameters that will be fixed in the
                next.
                Substituting \eqref{solu} in \eqref{nln} and using \eqref{pri} we obtain
                \begin{equation}
                    4^{\alpha}\left[\frac{\Gamma\left(\frac{\beta}{2}+1\right)}
                    {\Gamma\left(1-\alpha+\frac{\beta}{2}\right)}\right]^2 k
                    w^{\beta-2\alpha}= \lambda k^s w^{\beta s},
                \end{equation}
                that is satisfied when
                \begin{equation}
                    \begin{cases}
                        \beta = \frac{2\alpha}{1-s}\\
                        k =\left(\frac{4^{\alpha}}{\lambda}\left[\frac{\Gamma\left(1+\frac{\alpha}{1-s}\right)}
                        {\Gamma\left(1-\alpha+\frac{\alpha}{1-s}\right)}\right]^2\right)^{\frac{1}{s-1}},
                    \end{cases}
                \end{equation}
                as claimed.

                \end{proof}
                We observe that, for $s<1$ we have bounded solutions for $|x|\leq ct$, while for $s>1$
                the solutions are bounded for $|x|<ct$.

                We recall that similar specific solutions to the
                nonlinear Klein--Gordon equation (in the non fractional case) were
                investigated by \cite{Mat}. In particular, for $\alpha =1$, we recover the
                solution (4.7a) in the two-dimensional case (space and time), that is
                \begin{align}
                    u_1(x,t)=&\left(\frac{4}{\lambda}
                    \frac{1}{(s-1)^2}\right)^{\frac{1}{s-1}}\\
                    \nonumber&\times\left(c^2t^2- x^2
                    \right)^{1/(1-s)}.
                \end{align}
                Note that $c=1$ in the original paper \cite{Mat}.

                Similarly to Theorem 4.1, an exact
                solution of the non-homogeneous equation
                \begin{align}
                    \nonumber &\left(\frac{\partial^2}{\partial t^2}-c^2\frac{\partial^2}{\partial x^2}
                    \right)^{\alpha}u_{\alpha}(x,t)= \lambda u_{\alpha}^s (x,t)\\
                    \nonumber & \qquad \qquad + \gamma\left(c^2t^2-x^2\right)^{\alpha/(1-s)} ,\\
                    \nonumber & \qquad x\in \mathbb{R}, \:
                    t\geq 0, \: \alpha \in (0,1], n\neq 1, \gamma, \lambda \in \mathbb{R},
                \end{align}
                can be found.

                A case of particular physical interest is $s=3$, where
                \eqref{nln}, for $\alpha =1$, is strictly related to
                scalar $\phi^4$ theory. In the fractional case, we have
                the specific solution
                \begin{align}
                u_{\alpha}(x,t)=&\left(\frac{4^{\alpha}}{\lambda}\left[\frac{\Gamma\left(1-\frac{\alpha}{2}\right)}
                        {\Gamma\left(1-\frac{3}{2}\alpha\right)}\right]^2\right)^{\frac{1}{2}}\\
                        \nonumber&\times\left(c^2t^2- x^2
                \right)^{-\alpha/2},
                \end{align}
                which, for $\alpha =1$, becomes
                \begin{equation}
                u_1(x,t)= [\lambda (c^2t^2-x^2)]^{-1/2},
                \end{equation}
                that is the so-called meron solution in gauge theory.

        \bigskip

        \textbf{Acknowledgements}

        F.~Polito has been supported by project AMALFI (Universit\`{a} di Torino/Compagnia di San Paolo).

        We thank the anonymous reviewers for their accurate analysis of the first
        draft of the paper and for bringing some relevant papers
        to our attention.


\begin{thebibliography}{99}

        \bibitem{Bol} C.G. Bollini, J.J. Giambiagi.
            Arbitrary powers of D'Alembertians and
            the Huygens' principle, \emph{Journal of Mathematical Physics}, 34(2):610--621, (1993)

        \bibitem{Dim}
            I. Dimovski. On an operational calculus for a differential operator,
            \emph{Compt. Rendues de l'Acad. Bulg. des Sci.}, 21(6):513--516, (1966)

        \bibitem{us} R. Garra, E. Orsingher, F. Polito. Fractional Klein--Gordon equations and related
            stochastic processes, \emph{Journal of Statistical Physics}, in press,
            DOI: 10.1007/s10955-014-0976-0, (2014)

        \bibitem{Garra-FCAA}
            R. Garra, A. Giusti, F. Mainardi, G. Pagnini. Fractional relaxation
            with time-varying coefficient, \emph{Fractional Calculus and Applied
            Analysis}, 17(2), DOI: 10.2478/s13540-014-0178-0, (2014)

        \bibitem{Bale}
            A.K. Golmankhaneh, A. Golmankhaneh, D. Baleanu. On nonlinear fractional Klein--Gordon equation,
            \emph{Signal Processing}, 91:446--451, (2011)

        \bibitem{Kir-1994}
            V. Kiryakova. \emph{Generalized Fractional Calculus and Applications},
            Longman - J. Wiley, Harlow - N. York, (1994).

        \bibitem{Kir-2000}
            V. Kiryakova. Multiple (multiindex) Mittag--Leffler functions and
            relations to generalized fractional calculus,
            \emph{J. Comput. Appl. Mathematics}, 118:241--259, (2001)

        \bibitem{Kir-Luch}
            V. Kiryakova, Yu. Luchko. The multi-index Mittag--Leffler functions
            and their applications for solving fractional order problems in
            applied analysis, \emph{American  Institute of Physics - Conf. Proc.}
            1301:597--613, (2010)

        \bibitem{li} S.C. Lim, S.V. Muniandy.
            Stochastic quantization of nonlocal fields,
            \emph{Physics Letters A}, 324(5--6):396--405, (2004)

        \bibitem{main} F. Mainardi. \emph{Fractional Calculus and Waves in Linear Viscoelasticity},
            Imperial College Press, London, (2010)

        \bibitem{Mat} Y. Matsuno. Exact solutions for
            the nonlinear Klein--Gordon and Liouville equations in
            four-dimensional Euclidean space,
            \emph{Journal of Mathemathical Physics}, 28(10):2317--2322, (1987)

        \bibitem{mc} A.C. McBride. Fractional Powers of a Class of Ordinary Differential Operators,
            \emph{Proceedings of the London Mathematical Society}, 3(45):519--546, (1982)

        \bibitem{mc2} A.C. McBride. A theory of fractional integration for generalized functions,
            \emph{SIAM Journal on Mathematical Analysis}, 6(3):583--599, (1975)

        \bibitem{mc1} A.C. McBride. \emph{Fractional calculus and integral transforms
            of generalised functions}, Pitman, London, (1979)

        \bibitem{polo} W. Okrasi\'nski, L. Plociniczak. On fractional Bessel equation and
           the description of corneal topography,
           \emph{arXiv:1201.2526}, (2012)

        \bibitem{fca} W. Okrasi\'nski, L. Plociniczak. A note on
            fractional Bessel equation and its asymptotics, \emph{Fractional Calculus and Applied
            Analysis}, 16(3):559--572, (2013)

        \bibitem{pod} I. Podlubny. \emph{Fractional Differential Equations},
            Academic Press, New York, (1999)

        \bibitem{rosa} J.J. Rosales, J.F. G\'omez, M. Gu\'ia, V.I. Tkach.
            Fractional electromagnetic waves, \emph{IEEE Xplore Proceedings},
            10.1109/LFNM.2011.6144969 (2011)

        \bibitem{Lamb} S.E. Schiavone, W. Lamb.
            A fractional power approach to fractional
            calculus, \emph{Journal of Mathematical Analysis and
            Applications}, 149(2):377--401, (1990)

        \bibitem{tar} V.E. Tarasov. Fractional Integro-Differential Equations for Electromagnetic Waves
           in Dielectric Media, \emph{Theoretical and Mathematical Physics} 158(3):355--359, (2009)

    \end{thebibliography}
    \end{document}